\documentclass[a4paper,10pt]{article}
\usepackage{amsmath}
\usepackage{amsthm}
\usepackage{amsfonts}
\usepackage{amssymb}

\newtheorem{theorem}{Theorem}
\newtheorem{proposition}{Proposition}
\newtheorem{lemma}{Lemma}
\newtheorem{corollary}{Corollary}

\newcommand\Z{{\mathbb Z}}
\newcommand\Pp{{\mathbb P}}
\newcommand\Gm{{\mathbb G}_m}
\newcommand\ord{{\rm ord}}

\newcommand\C{{\mathbb C}}
\newcommand\OS{{\cal O}_S}
\newcommand\A{{\mathbb A}} 

\newcommand\cc{{\cal C}}

\newtheorem*{analogo}{Theorem}

\title{Integral points, divisibility between values of polynomials and entire curves on surfaces}
\author{Pietro Corvaja \&  Umberto Zannier}
\date{}

\begin{document}

\maketitle

\begin{abstract}
 We prove some new degeneracy results for integral points and entire curves on surfaces; in particular, we provide the first examples, to our knowledge, of a simply connected smooth variety whose sets of integral points are never Zariski-dense (and no entire curve has Zariski-dense image).
Some of our results are connected with divisibility problems, i.e. the problem of describing the integral points in the plane where the values of some given polynomials in two variables divide the values of other given polynomials.
\end{abstract}

\section{Introduction}
 
The purpose of this paper is to continue the investigation initiated in \cite{cz1} on integral points on surfaces, concentrating now on rational quasi projective surfaces. We shall in particular connect the distribution of integral points on such surfaces to divisibility problems: namely, given a set  of pairs of polynomials $(f_i(X,Y),g_i(X,Y))$ in two variables, we  shall study the set of integral points $(x,y)\in\A^2$ such that   $f_i(x,y)$ divides the value $g_i(x,y)$ of $g_i$. 
\smallskip

In the one-variable situation,  Siegel's finiteness theorem for integral points in the case of rational curves  can be restated by saying that: {\it if  $f(X),g(X)\not\equiv 0$ are coprime polynomials with coefficients in a ring of $S$-integers, such that for infinitely many $S$-integers $x$, $f(x)$ divides $g(x)$ in the ring of $S$-integers, then the polynomial $f(X)$ has at most one (complex) root.} 

One of the objects of the present work is to (partially) extend in dimension two this particular case of  Siegel's result. In this context, we mention the $S$-unit theorem in three variables, which can be reformulated as follows: {\it the set of pairs of $S$-integers $(x,y)$ such that $x,y$ and $x+y-1$ divide $1$ in the ring of $S$-integers is not Zariski-dense on the plane}. Our Theorem \ref{divisione} provides  a generalization of the $S$-unit  theorem to the case of more general polynomials. A very special case appeared already in \cite{cz2}, where we considered values of linear forms in three variables
dividing values of quadratic forms at $S$-unit  points.

As mentioned, our results can also be stated in terms of distribution of points on algebraic surfaces. In some cases we explicitly state a degeneracy result for integral points on  certain rational  surfaces, including for instance cubic hypersurfaces in $\Pp_3$ (see Theorem \ref{1}).
All our results constitute particular cases of the famous Vojta's conjecture, asserting: {\it let $\tilde{X}$ be a smooth projective variety, $D$ a hypersurface with  normal crossing singularities and $K$ a canonical  divisor for $\tilde{X}$. If $D+K$ is big, then no set of $S$-integal points on $\tilde{X}\setminus D$ is Zariski-dense.}

 As usual in this theory, after the fundamental work of Vojta \cite{vojta}, analogous results are expected for entire curves on surfaces: namely, on such a class of quasi projective varieties, no entire curve should be Zariski-dense. In some cases we shall also obtain such corresponding degeneration results for entire curves.
\medskip

The principal tool in this work will be the Main Theorem of   \cite{cz1}, which in turn is based on the Subspace Theorem of Schmidt and Schlickewei. As a consequence, our results will be ineffective; neverthless in some cases, where the set of integral points (resp. an entire curve) is infinite (resp. non-constant) but not Zariski-dense,  we shall find out the possible infinite families of integral points (resp. images of entire curves) lying on the given surface.

\medskip

Let us now state formally our main results. We follow the standard notation concerning $S$-integral points, as in \cite{cz1}, \cite{cz2} or \cite{bg}: for a number field $k$ and a finite set of places $S$, containing all the archimedean ones, let $\OS$ denote the corresponding  ring of $S$-integers, $\OS^*$ its group of $S$-units. Given a  projective variety $\tilde{X}$, embedded in a projective space $\Pp_N$, and a hypersurface $D\subset\tilde{X}$, both defined over $k$, let $X=\tilde{X}\setminus D$ be the complement of $D$ in $\tilde{X}$. We say that a rational point $P\in X(k)$ is $S$-integral if for no valuation $\nu$ of $k$ outside $S$, the reduction $P_\nu$ of $P$ modulo $\nu$ lies in the reduction of $D$ (we shall later define formally the precise notion of reduction of a sub-variety). Whenever $D$ is the intersection of $\tilde{X}$ with the hyperplane at infinity, so $X$ can be identified with a closed subvariety of $\A^N$, this definition coincides with the usual one.
\medskip

Some of our results are stated as degeneracy for integral points on surfaces, others as finiteness or degeneracy of the set of points satisfying some divisibility conditions. While introducing them, we shall present their   logical dependence; this will also be summarised at the end of this introduction.  
 
Our first result concerns cubic hypersurfaces of $\Pp_3$:

\begin{theorem}\label{1} 
 Let $\tilde{X}\subset \Pp_3 $ be a smooth cubic surface defined over $k$; let $H_1,H_2$ be hyperplane sections such that the curve $H_1\cup H_2$ consists of six lines. Then the $S$-integral points on the affine surface $X:=\tilde{X}\setminus (H_1\cup H_2)$ are not Zariski-dense. 
\end{theorem}

We recall that, after enlarging if necessary the field $k$, one can always find two hyperplanes sections on a smooth cubic surface $\tilde{X}$ consisting of six lines. Let us note that Vojta's conjecture for the specific case of cubic surfaces asserts the degeneracy of integral points after removing  at least two   hyperplanes sections; hence our result proves Vojta's conjecture in this case, up to the condition that both hyperplane sections be completely reducible. Theorem \ref{ctex1} will show that the number of lines to remove cannot be lowered, and that the condition that the six lines lie on two planes cannot be omitted.

A finiteness statement, in place of the degeneracy result of Theorem \ref{1}, cannot hold: actually, the smooth projective cubic surface $\tilde{X}$ contains twenty-seven lines (defined over a suitable extension of the number field $k$); after removing six of them lying on two planes, there still remain twenty-one lines on $X$, each having just two points at infinity. Hence, after enlarging if necessary the ring of $S$-integers, we can always obtain twenty-one infinite familes of integral points. One could prove that for a ``generic'' cubic hypersurface, there are no other infinite families, i.e. the one-dimensional component of the Zariski-closure of $X(\OS)$ is indeed the union of twenty-one lines.. 

A last consideration on the variety $X$ is in order: let $L_i (x_0:\ldots x_3)=0$ (for $i=1,2$) be linear equation for the hyperplane sections $H_i$ appearing in Theorem \ref{1}; then the rational function $f:=L_1/L_2$  on $\Pp_3$ provides a surjective morphism $X\rightarrow\Gm$. It can be easily proved (see Theorem \ref{Albanesecubiche}) that the torus $\Gm$ is in fact the generalized Albanese variety of $X$, and $f:X\rightarrow{\rm Alb}(X)=\Gm$ is the corresponding canonical map; in another language, the logarithmic irregularity of $X$ is $1$. While Faltings and  Vojta proved the degeneracy of integral points on surfaces whose (logarithmic) irregularity is at least three, this is one of the rare examples, to our knowledge, where degeneracy can be proved although the irregularity is only one. We shall even present a more striking example in Corollary \ref{rettescoppiate} to Theorem \ref{divisione}, of a surface where degeneracy of integral points still holds, although its Albanese variety is trivial.
\smallskip

Theorem \ref{1}  admits its natural counterpart in the complex analytic setting:

\begin{analogo}{\bf \ref{1} bis.}
 Let $\tilde{X}\subset\Pp_3$ be a smooth complex cubic surface, and   $X$, as before, the affine surface obtained from $\tilde{X}$ by removing six lines lying on two planes. For every  non costant holomorphic map $f:\C\rightarrow {X}$, the image $f(\C)$ is an algebraic curve.
\end{analogo}

Again, this result is best-possible, as shown by Theorem \ref{ctex1} bis. As for Theorem \ref{1}, for a general surface $X$ as above, one can prove that the image $f(\C)$ is in fact a line on $X$. 
\smallskip

Theorem 1 will be derived from a general statement concerning integral points on a suitable blow-up of the plane. The study of integral points on such rational surfaces, in turn, is motivated by  questions of divisibility among values of polynomials at integral points.

 As mentioned, a tipical example of divisibility problem is the $S$-unit equation in three variables, where one is interested in pairs of $S$-integers $(x,y)\in\OS^2$ such that the $S$-integers $x$, $y$ and $x+y-1$ divide  $1$, in the ring $\OS$ (i.e. they are $S$-units); it is well known that this problem can be reformulated in terms of integral points on the complement of four lines in general position in $\Pp_2$. 
A natural generalization consists of replacing the  constant polynomial by  arbitrary polynomials. In this direction we shall prove  the following result, for which we need a definition:
\medskip

\noindent{\bf Definition}. Let $m\geq 1$ be an integer, $(f_1,g_1),\ldots, (f_k,g_m)$ be pairs of polynomials in $k[X,Y]$. We say that the $m$ pairs above are in {\it general position} if the following conditions are satisfied:
\begin{itemize}
 \item for $1\leq i\neq j\leq m$, the curves of equation $f_i=0, f_j=0$ have no point in common at infinity (after embedding $\A^2\hookrightarrow\Pp_2$);
\item for $1\leq i<j<h\leq m$, the three affine curves $f_i=0, f_j=0, f_h=0$ have no point in common;
\item for each $i\in\{1,\ldots, m\}$ such that $g_i$ is non constant, the affine curves $f_i=0$ and $ g_i=0$ intersect transversely;
\item for $1\leq i<j\leq m$ and $h\in\{i,j\}$, the three curves $f_i=0$, $f_j=0$, $g_h=0$ have no point in common.
\end{itemize}

With this  notation we have:  

\begin{theorem}\label{divisione} 
 Let $(f_1,g_1), (f_2,g_2), (f_3,g_3)$ be three  pairs of nonzero polynomials in $\OS[X,Y]$;   suppose they are in general position in the above sense  and that 
$$
 \deg f_i\geq\max\{1,\deg g_i\}
$$ 
for $i=1,2,3$. Then the set of points $(x,y)\in\OS^2$ such that $f_i(x,y)|g_i(x,y)$ in $\OS$ is not Zariski-dense in $\A^2$.
\end{theorem}

In the case $\deg(f_i)=1,\, g_i\equiv 1$, the statement is equivalent to the $S$-unit theorem in three variables.   
Theorem \ref{divisione} is best possible, in the sense that its conclusion certainly does not hold if we just suppose $f_i(x,y)|g_i(x,y)$ for $i=1,2$: simply take $f_1=x, f_2=y, g_1=g_2=1$. 
 Also, it is easy to see that none of the four general position  conditions in the Definition can be omitted  (although they might probably be replaced by weaker ones). \medskip

Of course,   its analogue in complex analysis, which still holds, assures, under the above hypotheses for $f_i,g_i$, the algebraic dependence of any pair of entire functions $\varphi,\psi$ such that the three functions $g_i(\varphi,\psi)/f_i(\varphi,\psi)$ are holomorphic. This is an extension of Borel's theorem on pairs of entire functions $\varphi,\psi$ such that $1/\varphi, 1/\psi, 1/(\varphi+\psi-1)$ are holomorphic.
\medskip

A concrete example is the following very special case of Theorem \ref{divisione}, which is however still strong enough to imply the $S$-unit theorem in three variables; it is obtained by taking $g_1=g_2=g_3$, which is not excluded by the general assumption position:

\begin{corollary}
 Let $g(X,Y)\in\OS[X,Y]$ be a polynomial of degree $\leq 1$ such that $g(0,0)\neq 0, g(1,0)\neq 0, g(0,1)\neq 0$. The pairs $(x,y)\in\OS^2$ such that $xy(1-x-y)|g(x,y)$ are not Zariski-dense in $\A^2$.
\end{corollary}

More generally, the case of Theorem \ref{divisione} in which $f_1,f_2,f_3,g_1,g_2,g_3$ have all degree one seems worth being mentioned. It is equivalent to the following 

\begin{corollary}\label{rettescoppiate}
Let $L_1,\ldots,L_4$ be four lines in general position on the plane $\Pp_2$ and choose three  points  $P_1,P_2,P_3$, with $P_i\in L_i$ for $i=1,2,3$, $P_i\not\in L_j$ if $j\neq i$.  
 Let $\tilde{X}\to\Pp_2$ be the blow-up of the three points $P_1,P_2,P_3$ and let $D\subset \tilde{X}$ be the strict transform of $L_1+\ldots+L_4$.  Let $X:=\tilde{X}\setminus D$ be its complement. Then $X(\OS)$ is not Zariski-dense. 
\end{corollary}

The interest of this corollary lies in the geometrical fact:

\begin{theorem}\label{simplyconnected}
 The above defined affine surface $X$ is simply connected. In particular, its generalized Albanese variety is trivial. 
\end{theorem}

So, we have an example of a {\it simply connected smooth algebraic variety whose integral points, over any ring of $S$-integers, are never Zariski-dense}.  We think this is the first example of a simply connected variety for which such a statement has been proved. Since smooth simply connected curves are isomorphic to $\A^1$ or $\Pp_1$, there cannot exist examples in dimension one. Moreover, no example can be found by using the powerful technique of Faltings-Vojta  which, employing in an essential way the  (generalised) Albanese variety, cannot give anything for simply-connected varieties.
 
Of course, its analogue in Nevanlinna theory still holds: for every holomorphic map $f:\C\to X(\C)$, the image $f(\C)$ is contained in an algebraic curve. Again, this seems to be the first example of a simply connected algebraic variety for which it is known  that no entire curve is Zariski dense. By other methods,   one can prove that the ``very generic'' hypersurface in $\Pp_3$ of high degree is hyperbolic: for instance, in \cite{paun}  this fact is proved for hypersurfaces of degree 18. Hence, the {\it existence} of simply connected surfaces admitting no Zariski-dense entire curve was proved, even in the compact case, but no single explicit example was known. 
\medskip

A variation of Theorem \ref{divisione} concerns homogeneous forms in three variables instead of polynomials. For simplicity we state only a particular case of a much more general result which could be proved by our methods:

\begin{theorem}\label{forme}
 Let $F_1,\ldots,F_r$ be absolutely irreducible homogeneous forms in three variables, of the same degree, with coefficients in a ring of integers $\OS$. Let $G$ be another absolutely irreducible homogenous form in three variables, still defined over $\OS$.
Suppose that: (1) for every point $p$ where $F_i=G=0$ for some $i$, $p$ is a smooth point of both curves $F_i=0$ and $G=0$ and the corresponding tangents are distinct; (2)
 for distinct $1\leq i<j<h\leq r$, there is no non-trivial solution  to the equation $F_i=F_j=F_k=0$.  Suppose moreover that
$$
\deg(F_i)\geq \deg(G)\quad {\rm for}\quad i=1,\ldots,r\qquad {\rm and}\qquad r\geq 5.
 $$
Then the integral points $(x,y,z)\in\OS^3$ such that 
\begin{equation}\label{divisioneforme}
F_i(x,y,z)|G(x,y,z)\qquad{\rm for}\qquad i=1,\ldots,r,
\end{equation}
in the ring $\OS$, are not Zariski-dense in $\Pp_2$.
\end{theorem}

(By abuse of notation, we have identified the vector $(x,y,z)\in\OS^3$ with its corresponding point $(x:y:z)$ in $\Pp_2$; this will be possible since both sides on   (\ref{divisioneforme}) are homogeneous and $\deg F_i\geq \deg G$, so if $(x,y,z)$ is a solution of (\ref{divisioneforme}) and $\lambda\in\OS$ is a common divisor of $x,y,z$, then a fortiori $(x/\lambda,y/\lambda,z/\lambda)$ will be another solution). Again, the conditions on the degree and on the number of forms are sharp, as shown by Corollary \ref{ctforme} \S\ref{controesempi}.

\medskip

We shall show below another corollary, obtained from the case $\deg(F_i)=\deg(G)=2$ for all $i$. For this purpose, we recall some classical facts about  Del Pezzo surfaces of degree four (see e.g. \cite{beau}). A Del Pezzo surface of degree four can be obtained from the projective plane by blowing up  five points in general position: it can be embedded in a four-dimensional projective space, as a smooth complete intersection of two quadrics; let us denote it by $\tilde{X}\subset\Pp_4$. It contains exactly sixteen lines, and infinitely many conics. A hyperplane section of $\tilde{X}$ is  a quartic curve in a three-dimensional space, which will not be in general a plane curve; some special hyperplane sections, however, will be reducible, consisting of the union of two (plane) conics; it is easy to see that there are infinitely many such reducible hyperplane sections. Now, if   five generic hyperplane sections are removed from $\tilde{X}$, one obtains a closed (affine) variety $X\subset\Gm^4$ (since the complement in $\Pp_4$ of five hyperplanes in general position is isomorphic to $\Gm^4$), and an application of the $S$-unit theorem in several variables proves the degeneracy of its integral points. In the particular case where the hyperplane sections are reducible, we improve on such a  consequence of the $S$-unit theorem: as an application of Theorem \ref{forme} we shall obtain 

\begin{corollary}\label{DelPezzo}
 Let $\tilde{X}\subset\Pp_4$ be a Del Pezzo surface of degree four. Let $H_1,\ldots, H_5$ be five hyperplane sections, each of the form $H_i=\cc_i+\cc_i^\prime$, for smooth conics $\cc_i,\cc_i^\prime$. Suppose they are in general position, in the sense that no three of them intersect. Set $X:=\tilde{X}\setminus (\cc_1\cup\ldots\cup\cc_5)$. Then the integral points for the set $X(\OS)$ are not Zariski-dense. 
\end{corollary}

So, while the application of the $S$-unit theorem requires the removal of all the ten conics from $\tilde{X}$ to ensure degeneracy of integral points, Theorem \ref{forme} (via the above corollary) will provide the same conclusion after removing just five of them.

Actually we could prove that if the hyperplane sections are sufficiently generic (but still reducible), all but finitely many integral points lie in the union of up to eleven lines. Also, such infinite families cannot be avoid, since any such surface contains eleven lines with at most two points at infinity.

  Finally, let us note that the quasi-projective variety $X$ in Corollary \ref{DelPezzo} is never affine; as previously remarked, some questions about divisibility of values of polynomials in two variables amount to problems on distribution of integral points on surfaces which are neither projective nor affine.  
\medskip

Let us present a further result on Del Pezzo surfaces of degree four: recall that by  the $S$-unit equation theorem, the integral points on the complement of four pairwise linearly equivalent divisors on a surface are degenerate. Again, when one of them is reducible, we can improve on this classical result, obtaining:

\begin{theorem}\label{DelPezzo2}
 Let $\tilde{X}\subset\Pp_4$ be a Del Pezzo surface of degree four. Let $H_1,\ldots,H_4$ be hyperplane sections, such that no three of them intersect. Suppose one of them, say $H_4$, is reducible, so $H_4=\cc+\cc^\prime$, for two curves $\cc,\cc^\prime$ (either two conics, or a cubic and a line). Then the integral points on $X:=\tilde{X}\setminus(H_1\cup H_2\cup H_3\cup \cc)$ are degenerate.
\end{theorem}

Again, both results above admit the obvious analogue in complex function theory, i.e. the algebraic degeneracy of entire curves on $X$.
\medskip

As announced, Theorem 1 too is a corollary of a general result concerning divisibility of values of homogeneous forms at integral points. To state it, we shall adhere to the following standard notation concerning ideals in a ring: 
for  $S$-integers $\alpha_1,\ldots,\alpha_h\in\OS$, we shall denote by $(\alpha_1,\ldots,\alpha_h)$ the ideal they generate in the ring $\OS$. (We shall give soon-after another equivalent formulation, without mentioning ideals).

\begin{theorem}\label{ennagono}
 Let $n\geq 6$ be an integer, and for $i\in\Z/n\Z$, let $F_i$ be a linear form in three variables, such that any three of them are linearly independent. The set of points $(x:y:z)\in\Pp_2(k)$  satisfying,   for each $i\in\Z/n\Z$, the  equality of ideals
\begin{equation}\label{ideali}
F_i(x,y,z)\cdot(x,y,z)  =\big(F_{i-1}(x,y,z),F_{i}(x,y,z)\big)\cdot\big(F_{i}(x,y,z), F_{i+1}(x,y,z)\big) 
\end{equation}
is not Zariski dense in $\Pp_2$.
\end{theorem}

In the above formula, by homogeneity one can suppose that $x,y,z$ are $S$-integers; otherwise, one can interpret the ideals in the above equation in the sense of fractional ideals of $k$.\smallskip

Let us now see another way of stating Theorem \ref{ennagono}, without mentioning ideals. 
Enlarging $S$ by adding a finite number of places, we  can ensure that the ring $\OS$ is a unique factorization domain, so it makes sense to define the greatest common divisor of two elements $\alpha,\beta\in \OS$ (as a generator of the ideal $(\alpha,\beta)$). Also, we can suppose that any three of the linear forms are independent modulo any place outside $S$ (since by hypothesis they are independent over $k$). Then, for coprime $S$-integers $x,y,z$, any three of the values $\varphi_i(x,y,z)$ will be coprime. Let us put for every $i\in\Z/ n\Z$,
$$
\beta_i=\gcd(F_i(x,y,z), F_{i-1}(x,y,z));
$$
then for $i\neq j$, the $S$-integers $\beta_i,\beta_j$ are coprime, in particular $\gcd(\beta_i,\beta_{i+1})=1$ so  that we can write 
$$
F_i(x,y,z)=\beta_i\beta_{i+1}\alpha_i,
$$
for an $S$-integer $\alpha_i$. Then Theorem \ref{ennagono} asserts the degeneracy of the points $(x:y:z)\in\Pp_2(k)$ such that $\alpha_i$ is a unit for all $i\in\Z/n\Z$.
\medskip

We shall see that Theorem 1 follows formally from the case $n=6$ of Theorem \ref{ennagono}, and that this and all other results presented so far  can be interpreted as the degeneracy of integral points on certain blow-ups of the projective plane; these facts, in turn, consist  of particular cases of Vojta's conjecture for surfaces. 
On the other hand, the analogous result for $n\leq 5$ does not hold (see Theorem \ref{ctennagono}), in accordance to the fact that the hypotheses of Vojta's conjecture also fail  in that case.

We state below the complex-analytic analogue of Theorem \ref{ennagono}:

\begin{analogo}{\bf \ref{ennagono} bis.}
 Let $n\geq 6$ be an integer and for each $i\in\Z/n\Z$ let $\varphi_i$ be a linear form in three variables, with complex coefficients, such that no three of them are linearly dependent. Let $f,g,h:\C\rightarrow\C$ be three entire functions without common zeros. Suppose the following holds: for every index $i\in\Z/n\Z$ and every zero $p\in\C$ of the holomorphic function $\varphi_i(f,g,h)$, $p$ is also a zero of either $\varphi_{i-1}(f,g,h)$ or $\varphi_{i-1}(f,g,h)$ and
$$
\ord_{p}\, \varphi_i(f,g,h) =\max\big(\ord_p\, \varphi_{i-1}(f,g,h), \ord_p\,  \varphi_{i+1}(f,g,h)\big).
$$
Then the meromorphic functions $f/h, g/h$ are algebraically dependent.
\end{analogo}

Theorem \ref{ennagono} extends to forms $F_1,\ldots,F_n$ of arbitrary degree, provided some general position conditions are imposed.
\medskip

A last application of our methods concerns algebraic families of $S$-unit equa\-tions, of the kind treated by the authors in \cite{cz2} and by Levin in \cite{l2}; we fix three polynomials $f(T),g(T)$, $h(T)\in\OS[T]$ and consider the $S$-unit equation
\begin{equation}\label{sunits}
 f(t)u+g(t)v=h(t)
\end{equation}
to be solved in $S$-units $u,v$ and $S$-integers $t$. In \cite{cz2} the authors considered the linear case, where $\deg(f)=\deg(g)=\deg(h)=1$, while Levin treated in \cite{l2} the case when $\deg(f)+\deg(g)=\deg(h)$. In both cases, the statement was reduced to a question of integral points on the complement of a suitable divisor in the product $\Pp_1\times\Pp_1$. In the present work, we shall consider a different compactification, given by a Hirzebruch\footnote{The so-called Hirzebruch surfaces were introduced by Friederich Hirzebruch in [{\it Math. Annalen}, 1951]; later Aldo Andreotti realised that all minimal rational surfaces are of this kind.} surface (as it was done for a different diophantine problem in \cite{c}), and prove   

\begin{theorem}\label{teounits}
 Let  $f(T),g(T),h(T)$ be three polynomials of the same degree   with $S$-integral coefficients, without common zeros. There exists a finite set $\Phi\subset\OS^*$ such that for all solutions   $(t,u,v)\in\OS\times\OS^*\times\OS^*$ to equation (\ref{sunits}), one at least among $u,v,u/v$ belongs to $\Phi$. In particular, the solutions are not Zariski-dense in the surface defined by (\ref{sunits}). 
\end{theorem}
\smallskip

It is easily seen that for every zero $t_0$ of $f(t)g(t)h(t)$, there exists an infinite family of solutions, up to enlarging $S$; for instance, if $f(t_0)=0$, just put $v=-h(t_0)/g(t_0)$ and take for $u$ any $S$-unit. For such families, either $u$ or $v$ or $u/v$ is fixed. In some cases, there might also exist infinite families with variable $t$.

\medskip

The next statements will show that the integral points on certain open surfaces are indeed Zariski-dense: they will prove that our Theorems \ref{1},  \ref{forme}, \ref{ennagono}  are  in a sense best-possible. Several other statement of the same flavour will be given in section \ref{controesempi}.  

We begin by showing that in Theorems \ref{1}, \ref{1}bis  one cannot omit neither the condition on the number of lines to remove, nor the condition that they lie on two planes: 

\begin{theorem}\label{ctex1}(Counter-example)
Let $\tilde{X}$ be a smooth cubic surface, defined over a number field $k$; let $L_1,\ldots, L_5$ be five lines lying on two hyperplane sections of $\tilde{X}$; put $X=\tilde{X}\setminus(L_1\cup\ldots\cup L_5)$. There exists  a finite extension $k^\prime$ of $k$ and a finite set $S$ of places of $k^\prime$ such that $X(\OS)$ is Zariski dense. Also, there exists a configuration of nine lines $L^\prime_1,\ldots,L_9^\prime$ and a Zariski-dense set of $S$-integral points on $\tilde{X}\setminus(L_1^\prime\cup\ldots\cup L_9^\prime)$.
\end{theorem}
Its complex-analytic analogue reads:

\begin{analogo}{\bf \ref{ctex1} bis.}
 Let $\tilde{X}$ be a smooth cubic surface, $L_1,\ldots, L_5$ be five lines on $\tilde{X}$, lying on two hyperplanes; there exists a holomorphic map $f:\C\rightarrow \tilde{X}$ omitting all the lines  $L_1,\ldots,L_5$  and having Zariski-dense image. Also, there exists a configuration of nine lines $L^\prime_1,\ldots,L_9^\prime$ and a holomorphic map $g:\C\rightarrow\tilde{X}$ omitting $L_1^\prime,\ldots,L_9^\prime$ and having Zariski-dense image.
\end{analogo}

This counter-example can be strengthened, using the work of Buzzard and Lu \cite{bl}; we could even prove that the complement of five lines in a smooth cubic surface is holomorphically dominated by $\C^2$.  

We now show that Theorem \ref{ennagono} is also best-possible, in the sense that it fails in the case $n\leq 5$:

\begin{theorem}\label{ctennagono}
 Let $n$ be an integer with $1\leq n\leq 5$;  let $F_1,\ldots,F_n$ be linear forms in three variables in general position, defined over a number field $k$. There exists a ring of $S$-integers $\OS$, for a finite set of places $S$, such that the set of integers $(x,y,z)\in\OS^3$ satisfying the relation (\ref{ideali}) is Zariski dense (in $\Pp_2$ or, equivalently, in $\A^3$).
\end{theorem}

To summarise: we dispose of three main ``non-divisibility theorems'',  namely Theorems 2, 4, and 6, stating that certain divisibility conditions are satisfied only in proper Zariski-closed sets of the plane. These results formally imply some theorems on integral points on surfaces: Corollary 1 (on the simply connected surface) follows from Theorem 2; Corollary 2 (on Del Pezzo surfaces of degree four) from Theorem 4; Theorem 1 (on cubic surfaces) follows from Theorem 6. 

One extra result (Theorem 5) is   stated in terms of $S$-unit equations with parametric coefficients.

The rest of the paper is organized as follows: In the next section we first show how divisibility and integrality on quasi-projective surfaces are related, and then prove our main diophantine theorems, the degeneracy of integral points on certain surfaces.

In \S \ref{geometrica} we prove our claims of geometric type, for instance the simply-con\-nected\-ness of the surface appearing in Corollary \ref{rettescoppiate}, and the determination of the Albanese variety of the surface appearing in Theorem 1.  We also classify the infinite families of solutions to the $S$-unit equation (\ref{sunits}).

Finally, the last section contains results in the opposite direction: we shall prove that integral points on certain surfaces form a Zariski-dense set, thus proving that some of our degeneracy results are in a sense best-possible: for instance, one cannot omit the hypothesis appearing in Theorem \ref{1} that the six lines at infinity lie on two hyperplanes, and one cannot improve on the hypothesis $n\geq 6$ in Theorem \ref{ennagono}.

\section{Proofs of main results}\label{dimostrazioni}

In this section, we shall prove the degeneracy of integral points on the surfaces considered in the introduction,  the  Theorems \ref{divisione} and \ref{forme} on divisibility and Theorem \ref{teounits} on parametric $S$-unit equations. We shall also show the relations between them. The corresponding results 
on entire curves can be proved in the same way, by replacing Schmidt's Subspace Theorem by Cartan's second Main Theorem, as explained by Levin \cite{l1}. Hence, we shall omit the proofs in the analytic setting, and work only over number fields.

In the sequel, $k$ will be a fixed number field, $S$ a finite set of places of $k$ containing the archimedean ones, $\OS$ (resp. $\OS^*$) will denote the ring of $S$-integers (resp. group of $S$-units) of $k$. We recall that for every non-archimedean place $\nu$ of $k$, with residue field $k_\nu$, there is a well-defined reduction map $\Pp_n(k)\rightarrow\Pp_n(k_\nu)$.
In particular, for a projective variety $\tilde{X}\subset\Pp_n$, defined over $k$, and a rational point $P\in\tilde{X}(k)$, we can speak of the reduction of $P$ modulo $\nu$. 
If $D\subset\tilde{X}$ is a closed subvariety, we can always define in two different ways its reduction modulo $\nu$ (which will be a variety defined over the residue field $k_\nu$ of $\nu$) as follows: (1)  Set-theoretically: extend in some way the valuation $\nu$ to the algebraic closure $\bar{k}$ of $k$ and consider the set of the reduction of the points $P\in D(\bar{k})$; it is the set of points $\bar{k_\nu}$-rational of an algebraic  variety $D_\nu$ defined over $k_\nu$. (2) Alternatively, note that all but finitely many places $\nu$, one can simply reduce modulo $\nu$ the coefficients of a given system of equations for $D$, obtaining a system of equations for $D_\nu$. Since we shall consider only finitely many subvarieties $D\subset \tilde{Y}$, and in each of our statement we are allowed to disregard a fixed but arbitrary finite set of places of $k$, one can always refer to the reduction modulo $\nu$ in the second  sense. 

Given a hypersurface $D\subset\tilde{X}$ defined over $k$, we say that a rational point $P\in\tilde{X}\setminus D(k)$ is $\nu$-integral if its reduction modulo $\nu$ does not lie in the reduction of $D$ modulo $\nu$; we say it is $S$-integral if it is $\nu$-integral for every place $\nu$ {\it outside} $S$.
\smallskip

The main tool for reducing questions of divisibility between values of polynomials to integrality for points on rational varieties is represented by the following lemma:

\begin{lemma}\label{riduzione}
 Let $\tilde{Y}\subset\Pp_N$ be a smooth projective surface defined over $k$, $\cc_1,\cc_2$ be curves on $\tilde{Y}$ intersecting transversely at points $P_1,\ldots,P_n$, also defined over $k$. Let $\varphi,\psi$ be rational functions on $\tilde{Y}$, defined over $k$, such that local equations (on an affine open set $U\subset\tilde{Y}$) for $\cc_1$ (resp. $\cc_2$) are given by $\varphi=0$ (resp. $\psi=0$). Let $\pi:\tilde{X}\rightarrow\tilde{Y}$ be the blowup of $\tilde{Y}$ over the points  $P_1,\ldots,P_n$;   for $i=1,2$ denote by $\hat{\cc_i}$,  the strict transform of $\cc_i$ under $\pi$. Let $\nu$ be a non-archimedean place of $k$ such that the following conditions are satisfied:

\begin{itemize}
\item{  the reductions modulo $\nu$ of $\cc_1,\cc_2$ intersect transversaly;}

\item{ the  reductions modulo $\nu$ of the functions $\varphi,\psi$ induce local equations for the reductions of $\cc_1,\cc_2$;}

\item{ the blow-up map $\pi$ induces an isomorphism modulo $\nu$ of the complement of the exceptional divisors in $\tilde{X}$ with the complement of $\{P_1,\ldots,P_n\}$ in $\tilde{Y}$.}
\end{itemize}

\noindent Then for every point $P\in\subset\tilde{Y}(k)$ lying in the domain $U$ of $\phi$ and $\psi$, not in $\cc_1\cup\cc_2$, the following are equivalent:
\begin{enumerate}
 \item $\nu(\varphi(P))\leq\nu(\psi(P))$

\item $\pi^{-1}(P)$ is $\nu$-integral with respect to $\hat{\cc_1}$.
\end{enumerate}

\end{lemma}

\noindent{\it Proof}. Let us suppose that 1.) holds; we want to prove that $\pi^{-1}(P)$ is $\nu$-integral with respect to $\hat{\cc_1}$. If $P$ does not reduce to $\cc_1$ modulo $\nu$, we are done (in fact $P$ is $\nu$-integral with respect to the divisor $\pi^*(\cc_1)$, which contains the strict transform $\hat{\cc_1}$). Then we can suppose that $\nu(\varphi(P))>0$;  by hypothesis 1.) we also have $\nu(\psi(P))>0$, so $P$ reduces modulo $\nu$ to an intersection point of $\cc_1\cap \cc_2$, say $P_i$. By the hypotheses  of the Lemma, the functions $\varphi,\psi$ are local parameters on the surface $\tilde{Y}$ at the point $P_i$. The blown-up surface can be defined locally at $P$ as the closed subset of $\tilde{Y}\times\Pp_1$ defined by the equation $\varphi(y)\xi_i=\psi(y)\eta_i$; the exceptional divisor $E_i$ is defined by $\varphi=\psi=0$ and $\hat{\cc_1}$ by $\eta_i=0$. We have to prove that $\pi^{-1}(P)$ does not reduce to $\eta=0$ modulo $\nu$. From the equality $\nu(\varphi(P))+\nu(\xi_i)=\nu(\psi(P))+\nu(\eta_i)$ and the inequality 1. it  follows that $\nu(\xi_i)\geq \nu(\eta_i)$, ie. $(\xi_i:\eta_i)$ is not congruent to $(1:0)$ modulo $\nu$, which, as remarked, is what we want. The other implication, which will not be used in the proof, can be proved by the same reasoning.
\qed

\medskip

We now prove Theorem \ref{ennagono}, from which we shall  deduce the first conclusion of Theorem \ref{1}.
Theorem \ref{ennagono} is equivalent to the following   

\begin{proposition}\label{ennagono2}
 Let $n\geq 6$ be an integer; for every index $i\in\Z/n\Z$, let  $H_i\subset\Pp_2$ be a line defined over $k$. Suppose that no three of them intersect. Let, for each index $i\in\Z/n\Z$, $P_i$ be the intersection point $H_i\cap H_{i+1}$. Let $\tilde{X}\rightarrow\Pp_2$ be the blow up of the plane over the points $P_1,\ldots,P_n$ and let $D_i\subset \tilde{X}$ be the corresponding strict transform of $H_i$. Finally put $D=D_1+\ldots+D_n$. Then no  set of $S$-integral points of $\tilde{X}\setminus D$ is Zariski-dense.
\end{proposition}

\noindent{\it Proof}. We apply the Main Theorem of \cite{cz1}, in its generalized form given in \cite{cz2}, Theorem 2.1, hence we follow the notation of \cite{cz1}, \cite{cz2}. We take $r=n$ and $p_i=1$ for all $i=1,\ldots,n$. To check the validity of the hypothesis of the  Theorem 2.1 in \cite{cz2}, we have to compute the intersection matrix of the divisor $D=D_1+\ldots+D_n$. For this purpose, note that since each divisor $D_i$ is the strict transform of the line $H_i$ in $\Pp_2$, which passes through two blown-up points, its self intersection equals $1-2=-1$. Also, for $i\neq j$, $D_i$ intersects $D_j$ in one point unless $i,j$ are consecutive. Finally we have, for all $i\in\Z/n\Z$,  $D_i^2=-1$ , $D_i. D_{i+1}=0$; $D_i.D_j=1$ for $j\neq i,i+1,i-1$. Then $D.D_i=n-4$ and $D^2=n(n-4)$, so $D^2>0$, and $D.{\cal C}\geq 0$ for all irreducible curves $\cc$ in $\tilde{X}$, as required by the hypotheses of Theorem 2.1 in \cite{cz2}. We note that for a generic choice of the lines $H_i$ in the plane, $D$ will  be ample, but in some degenerate case it will not be so (for instance, if all the $n$ intersection points $H_i\cap H_{i+1}$ lie on a conic $\cc$, then $D.\cc=0$; of course this does not happen generically since $n\geq 6$). The equation for $\xi_i=: \xi$ becomes
\begin{equation}\label{xi}
\xi^2+2(n-4)\xi-n(n-4)=0.
\end{equation}
The inequality $2D^2\xi>D.D_i\xi^2+3D^2$, appearing in the hypothesis of the Main Theorem of \cite{cz1} (or Theorem 2.1 of \cite{cz2}), now reads 
\begin{equation*}
 2n(n-4)\xi>(n-4)\xi^2+3n(n-4).
\end{equation*}
After simplifying the common factor $(n-4)$ and using the equation (\ref{xi}) for expressing $\xi^2$ as $\xi^2=n(n-4)-2(n-4)\xi$, the  above inequality becomes
$ 2(2n-4)\xi>n^2-n,$ i.e.
\begin{equation*}
 \xi>{n^2-n\over 2(2n-4)}.
\end{equation*}
We have to prove that this inequality  is satisfied by the minimal positive solution (actually the only positive solution) of (\ref{xi}). Now, this amounts to prove that the value of the polynomial on the left side in (\ref{xi}), calculated at the point $(n^2-n)/2(2n-4)$, is negative, i.e.
$$
\left({n^2-n\over 2(2n-4)}\right)^2+2(n-4){n^2-n\over 2(2n-4)}<n(n-4).
$$
The above inequality simplifies to $7n^3-70n^2+207n-192>0$, which is easily seen to be equivalent (for positive integers $n$) to $n\geq 6$, conlcuding the proof.
\qed
\medskip

\noindent {\it Deduction of Theorem \ref{ennagono} from Proposition \ref{ennagono2}}.
Let $F_1,\ldots,F_n$ be the linear forms appearing in Theorem \ref{ennagono} and let, for $i\in\Z/n\Z$,  $H_i\subset \Pp_2$ be the line of equation $ F_i(x,y,z)=0$. By the hypotheses of Theorem \ref{ennagono} such lines are in general position. Let $\pi:\tilde{X}\rightarrow \Pp_2$ be the blow up defined in the above Proposition, $E_1,\ldots,E_n$    the corresponding exceptional curves, $D$ the  hypersurface defined in the statement of Proposition \ref{ennagono2}.  Note that $\pi$ is an isomorphism between $\tilde{X}\setminus (D\cup E_1\ldots\cup E_n)$ and $\Pp_2\setminus (H_1\cup\ldots\cup H_n)$.
As we mentioned in the introduction, we can enlarge the set $S$ of places in such a way that the reductions of $H_1,\ldots,H_n$ modulo  of every place outside $S$ are still in general position. Also, we can suppose that the isomorphism between the affine surfaces $\tilde{X}\setminus (D\cup E_1\ldots\cup E_n)$ and $\Pp_2\setminus(H_1\cup\ldots\cup H_n)$, given by the restriction of $\pi$, is defined over $\OS$, together with its inverse, so it induces an isomorphism modulo every place outside $S$.

Let now $(x,y,z)\in\OS^3$ be a point satisfying the relation (\ref{ideali}) and such that $F_i(x,y,z)\neq 0$ for all $i\in\Z/n\Z$. If we prove that such a point defines an $S$-integral point $\pi^{-1}(x:y:z)$ on $\tilde{X}$ with respect to $D$, we have finished the proof that Proposition \ref{ennagono2} implies Theorem \ref{ennagono}.  

To this end, we use Lemma \ref{riduzione}. Let $\nu$ be a place of $k$ outside $S$, fix an index $i\in\Z/n\Z$ and a point $P=(x:y:z)\in\Pp_2(k)$ satisfying (\ref{ideali}); we have to prove that the reduction of $\pi^{-1}(P)$ modulo $\nu$ does not lie on $D_i$. If $F_i(P)\not\equiv 0$ (mod $\nu$), then we have finished, since $\pi^{-1}(P)$ is integral even with respect to the whole pull-back $\pi^*(H_i)$. Then suppose $F_i(P)\equiv 0$ modulo $\nu$. By the relation (\ref{ideali}), there exists an index $j\in\{i+1,i-1\}$ such that $\nu(F_i(P))=\nu(F_j(P))$. Suppose for instance $j=i+1$; this means that $P$ reduces to $H_i\cap H_{i+1}$ modulo $\nu$. By the hypothesis that the forms $F_1,\ldots,F_n$ are in general position modulo $\nu$, for all  other indices $h\in\{1,\ldots,n\}\setminus\{i,i+1\}$, we have $\nu(F_h(P))=0$. Let us choose one such index, say $i-1$, and consider the rational functions $\varphi:=F_i/F_{i-1},\psi:=F_j/F_{i-1}=F_{i+1}/F_{i-1}$. Then  local equations for $H_i$ (resp. $H_{i+1}$) in a neighborhood of $H_i\cap H_{i+1}$ are given by $\varphi=0$ (resp. $\psi=0$). 
Lemma \ref{riduzione} applies, proving that $\pi^{-1}(P)$ is integral with respect to $D_i$; on the other hand, it is clear that it is integral with respect to $D_j$ for every $j\neq i$, due to our initial hypothesis that the lines $H_1,\ldots,H_n$ are in general position modulo every place outside $S$;  hence $\pi^{-1}(P)$ is integral with respect to the whole divisor $D$ and the proof is complete.
\qed 

\medskip

\noindent {\it Proof of Theorem \ref{1}}. We suppose the hypotheses of Theorem \ref{1} hold; in particular, $\tilde{X}\subset\Pp_3$ is a cubic surface defined over a number field, $H_1,H_2$ are two completely reducible hyperplane sections. It is well known (see \cite{beau}) that $\tilde{X}$ can be obtained by blowing up the plane at six points $P_1,\ldots,P_6$ in general position (no three on a line, no six on a conic); after enlarging the field of definition, we can suppose that such points are defined over the given number field $k$. The configuration of six lines lying on two planes can be obtained as the strict transform of the six lines on the plane connecting $P_i,P_{i+1}$ for $i\in\Z/6\Z$. Then apply Theorem \ref{ennagono} to obtain the degeneracy of the set of integral points considered in Theorem \ref{1}.  
\qed
\medskip

\noindent{\it Proof of Theorem \ref{teounits}}. Let us consider the affine surface defined by equation (\ref{sunits}); it admits a smooth compactification in $\Pp_2\times\Pp_1$ obtained in the following way: write $\tilde{f},\tilde{g},\tilde{h}\in k[T_0,T_1]$ for the forms of degree $d=\deg(f)=\deg(g)=\deg(h)$ satisfying $\tilde{f}(1,T)=f(T),\, \tilde{g}(1,T)=g(T),\,  \tilde{h}(1,T)=h(T)$. Then our surface is an open affine set of the surface defined in $\Pp_2\times\Pp_1$, with coordinates $((U:V:W),(T_0:T_1))$, by the bi-homogeneous equation
$$
\tilde{X}:\qquad U\cdot \tilde{f}(T_0,T_1)+V\cdot \tilde{g}(T_0,T_1)=W\cdot \tilde{h}(T_0,T_1).
$$ 
It is endowed with a canonical projection $\tilde{X}\rightarrow \Pp_1$ whose fibers are projective lines. It turns out to be isomorphic to the $d$-th Hirzebruch surface (see \cite{beau}), independently of the degree $d$ polynomials $f,g,h$ (provided they have no common zero, which is assumed here).
The solutions $(u,v,t)\in\OS^*\times\OS^*\times\OS$ to equation (\ref{sunits}) give rise to rational points $((u:v:1),(1:t))$ which are integral with respect to the   divisor  $D:\quad T_0\cdot UVW=0$.
It consists of four components $D_1,D_2,D_3$ and $D_4$, where  $D_4$ is given by $T_0=0$, so it is a fiber for the projection $\tilde{X}\rightarrow \Pp_1$, while $D_1,D_2,D_3$ are pull-backs of  lines in $\Pp_2$ (via the natural projection $\tilde{X}\rightarrow \Pp_2$).
Then $D_4^2=0$, $D_1,D_2,D_3$ are linearly equivalent and satisfy $D_i.D_j=d\geq 1$ for $1\leq i,j\leq 3$. We are in the situation of applying Theorem 1.1 of \cite{cz2} (or the Main Theorem of \cite{cz1}) to obtain the degeneracy of the set of solutions. Hence all but finitely many solutions lie on a finite set of curves, which, by Siegel's theorem on integral points on curves, are parametrised either by $\A^1$ or by $\Gm$. It is easy to see that there are no curves parametrised by $\A^1$ on the affine surface $\tilde{X}\setminus D$. Those parametrised by $\Gm$, i.e. of vanishing Euler characteristic, are classified by Theorem \ref{linesHirzebruch}; they give rise to the infinite families with fixed $u$, $v$ or $u/v$, and this completes the proof.\qed
\medskip

To prove Theorem  \ref{divisione}, we need yet another consequence of our Main Theorem from \cite{cz1}, which we  immediately state and prove:

\begin{proposition}\label{proposizione2}
 Let $\tilde{X}$ be a smooth projective surface, $D_1,\ldots,D_r,H$ be reduced and irreducible divisors on $\tilde{X}$, no three of them intersecting. Suppose there exist positive integers $p_1,\ldots,p_r,c,h$  such that for $1\leq i<j\leq r$, $p_iD_i.H=ch$, $p_ip_j D_i.D_j=c^2$ and $H^2=h^2$. Suppose moreover that $D_i^2=0$. If $r\geq 3$, then the integral points on $\tilde{X}\setminus(H\cup D_1\cup\ldots\cup D_r)$ are not Zariski-dense.
\end{proposition}

\noindent{\it Proof}. It is clear that the result follows from the particular case when $r=3$, so we shall suppose $r=3$. Again, we  apply the Main Theorem of \cite{cz1}. Let us denote by $D_{4}$ the divisor $H$ and let $p_{4}=p$ be a positive integer to be chosen later. Put 
\begin{equation*}
 D=p_1D_1+\ldots+p_4D_4=p_1D_1+p_2D_2+p_3D_3+ pH.
\end{equation*}
We shall now verify that the hypotheses of the Main Theorem in \cite{cz1} are satisfied.
First of all, let us remark again that the condition that $D$ is ample, appearing in the statement of the Main Theorem in \cite{cz1} (and which, in any case, will  be satisfied in the applications to Theorem \ref{divisione} and \ref{forme}) is not really needed, and can be replaced by the condition that $D^2>0$ and $D$ is nef; this fact appears explicitly in \cite{cz2} and \cite{l1}. 
 Let us compute the intersection products. For $i\leq 3$ we have:
\begin{equation}\label{dd_i}
 D.D_i=\frac{2c^2}{p_i}+\frac{p ch}{p_i},
\end{equation}
while for the intersection number $D.D_{4}=D.H$ we have
\begin{equation}\label{dh}
 D.H=3ch+p h^2.
\end{equation}
Then 
\begin{equation}\label{d^2}
 D^2=6c^2+ 6 pch+p^2h^2.
\end{equation}
For $i\leq 3$, the equation for $\xi_i$, i.e. $(D-\xi_i D_i)^2=0$, gives
\begin{equation*}
 \xi_i=\frac{D^2}{2(D.D_i)}.
\end{equation*}
The inequality $2\xi_i D^2>D.D_i \xi_i^2+3p_i D^2$, appearing in the hypotheses of the Main Theorem of \cite{cz1} (or Theorem 2.1 of \cite{cz2}), becomes
\begin{equation*} 
 2 D^2\frac{D^2}{2(D.D_i)}>(D.D_i)\left(\frac{D^2}{2(D.D_i)}\right)^2+3p_i D^2,
\end{equation*}
which simplifies to $D^2>4p_i (D.D_i)$, i.e. $6c^2+ 6 pch+p^2h^2   >4(2c^2+pch) $, or
\begin{equation}\label{inD_i}
p^2h^2+2pch>2c^2.
\end{equation}
For $i=4$, the equation for $\xi=\xi_4$ is
\begin{equation}\label{xi_4}
 h^2\xi^2-2(D.H)\xi+D^2=0,
\end{equation}
which provides the relation 
\begin{equation}\label{xiquadro}
 \xi^2=(2(D.H)\xi-D^2)/h^2.
\end{equation}
 Solving the equation (\ref{xi_4}) for $\xi$, taking into account that by (\ref{dh}) and (\ref{d^2}) one has  $(D.H)^2-D^2H^2=3c^2h^2$, we obtain for the minimal solution $\xi$ the expression
\begin{equation}\label{xi4}
 \xi=\frac{3ch+ph^2-\sqrt{3}ch}{h^2}=(3-\sqrt{3})\frac{c}{h}+p.
\end{equation}
Now, the condition $2\xi D^2>(D.H)\xi^2+2pH^2$, appearing in the hypotheses of the Main Theorem of \cite{cz1}, becomes, after substituting $\xi^2$ by its expression in (\ref{xiquadro}), 
\begin{equation*} 
 2\xi D^2>(D.H)\left(\frac{2(D.H)\xi-D^2}{h^2} \right)+3pD^2.
\end{equation*}
The above inequality becomes $2\xi ((D.H)^2 -D^2 h^2 )<(D.H)D^2- 3p h^2 D^2$, which, by (\ref{xi4}), (\ref{dh}), (\ref{d^2}) and the already used fact that the discriminant  $(D.H)^2-D^2 H^2$ equals $3c^2h^2$, amounts to   
\begin{equation}\label{inH}
 6c^2h \left[(3-\sqrt{3})c+ph\right]<(3ch-2ph^2) (6c^2+6pch+p^2h^2).
\end{equation}
Remember that we were searching a positive integer $p$ such that both inequalities (\ref{inD_i}) and (\ref{inH}) are satisfied. Now, it is easy to see that such inequalities are indeed satisfied for $p=\frac{3c}{4h}$. This number might be a rational non integral one; but, from the relation $p_iD_iH=ch$, appearing in the present assumptions, and the fact that $h=(H^2)^{1/2}$ is fixed (i.e. independent of $p_1,p_2,p_3$), one easily see that if $p_1,p_2,p_3$ are replaced by suitable positive multiples, the assumptions of the theorem will still be satisfied and one can secure that the number $c/(4h)$ is indeed an integer. This remark concludes the proof.
\qed

\medskip
{\it Proof of Theorem \ref{divisione}}. Let $(f_1,g_1),(f_2,g_2), (f_3,g_3)$ be the  pairs of polynomials appearing in Theorem \ref{divisione}, and let $D_i^\prime$ (for $i=1,2,3$) be the curves in $\Pp_2$ whose local equations in $\A^2$ are $f_i=0$. Let $H^\prime$ be the line at infinity of $\Pp_2$. For each $i=1,2,3$, let  $\{P_{i,1},\ldots,P_{i,l_i}\}$ be the intersection of the affine curves $f_i=0$ and $g_i=0$; note that $l_i\leq \deg(f_i)^2$, by the assumption $\deg(f_i)\geq \deg g_i$. If $\deg g_i<\deg f_i$ complete the set $\{P_{i,1},\ldots,P_{i,l_i}\}$ by adding points $P_{i,l_i+1},\ldots, P_{i,\deg(f_i)^2}$ on the curve $f_i=0$, but outside the two other curves $f_j=0$ (for $1\leq j\leq 3$, $j\neq i$). We thus obtain $\sum_{i=1}^3\deg(f_i)^2$   points on the affine plane $\A^2$. 
The general position assumptions guarantee they are indeed distinct; also, they guarantee that no three among the four projective curves $D_1^\prime,D_2^\prime,D_3^\prime,H^\prime$ intersect and no two intersect on any of the points $P_{i,j}$. Let $\tilde{X}$ be the blow-up of the projective plane over the points $P_{i,j}$; let $D_i$ be the strict transform of $D_i^\prime$, and let $H\subset\tilde{X}$ be the curve corresponding to $H^\prime$. After enlarging if necessary the finite set of places $S$, we can ensure that at each point $P_{i,j}$ the assumptions of Lemma \ref{riduzione} are satisfied (with $\varphi=f_i, \psi=g_i$). Suppose now $(x,y)\in\OS^2$ is a point satisfying $f_i(x,y)|g_i(x,y)$ for $i=1,2,3$ (so in particular $f_i(x,y)\neq 0$, which excludes that $(x:y:1)$ is one of the blown-up points). Let $p\in\tilde{X}$ be its corresponding point on the blown-up surface. Since $x$ and $y$ are $S$-integers, the point $(x:y:1)$ is integer with respect to $H^\prime$, so the corresponding point $p$ is  integer with respect to $H$.  By Lemma \ref{riduzione}, the point $p$ is  integer also with respect to $D_i$ for $i=1,2,3$. We now show that we can apply Proposition \ref{proposizione2} to conclude that such points are not Zariski-dense.  Letting $d_i>0$ be the degree of $f_i$, we observe that $D_i$ is the strict transform of a curve of degree $d_i$ in $\Pp_2$ blown-up at $d_i^2$ points, so its self-intersection vanishes; the intersection product $D_iD_j$ (for $1\leq i<j\leq 3$) equals $d_id_j$, since no common point of $D_i^\prime, D_j^\prime$ is blown-up. Also, since no point in $D_i^\prime\cap H^\prime$ is blown-up, we have $D_i.H=d_i$. Now put $p_i=d_1d_2d_3/d_i$. It is immediate to check that the assumptions of Proposition \ref{proposizione2} are satisfied with $c=d_1d_2d_3$ and $h=1$. 
\qed
\medskip

The above proof also shows the correspondence between the divisibility problems considered in Theorem \ref{divisione} and integrality on certain surfaces; in such a correspondence, the case $\deg f_i=\deg g_1 =1$ for $i=1,2,3$ leads to the surface considered in the Corollary to Theorem \ref{divisione}.

\medskip

\noindent{\it Proof of Theorem \ref{forme}}. The proof is very similar to the above one and to the proof of Theorem \ref{ennagono}. Let, for each $i=1,\ldots,3$, $D_i^\prime\subset\Pp_2$ be the projective curve of equation $F_i=0$, and let $d=\deg F_i=\deg D_i^\prime$. Let $P_{i,1},\ldots,P_{i,d\cdot\deg G}$  be the intersection points of $D_i^\prime$ with the curve $G=0$ (by assumption  there are exactly $d\cdot\deg G$ of them). 
Let $\tilde{X}\rightarrow Pp_2$ be the blowu-up of the plane over all these points. Again by Lemma 1, we are reduced to proving the degeneracy of integral points on $ \tilde{X}\setminus(D_1\cup\ldots\cup D_r)$, where $D_i$ is the strict transform of $D_i^\prime$ for $i=1,\ldots,r$. In the case $\deg G=\deg F$, all the self products $D_i^2$  vanish, and the products $D_iD_j$ for $i\neq j$ are all equal to $d^2>0$. Then the degeneracy of integral points follows directly from Theorem 1, part (b), of \cite{cz1}. Hence the Theorem is proved in this case. If $\deg F>\deg G$, just remark that the theorem becomes weaker if we blow-up more points on the curves $D_i^\prime$; hence, let us choose, for each $i$, any set of $\deg F_i(\deg F_i-\deg G)$ points on the curve $D_i^\prime$, outside each the curve $D_j^\prime$ (for $j\neq i$) and the curve $G=0$. Considering the corresponding surface $\tilde{X}$, and letting again $D_i$ be the strict transforms of $D_i^\prime$, we reduce to the previous case. 
\qed

\medskip

\noindent {\it Proof of Corollary \ref{DelPezzo}}. We first derive  Corollary \ref{DelPezzo} from Theorem 1 of \cite{cz1}: the five conics $\cc_1,\ldots,\cc_5$ have self-intersection zero and the intersection products $\cc_i.\cc_j$, for $i\neq j$ are all equal to one, so Theorem 1 of \cite{cz1} prove the Corollary.
 
\smallskip

 We also sketch a deduction of the Corollary from Theorem \ref{forme}. Let $\tilde{X}\subset\Pp_4$ be a (smooth) Del Pezzo surface of degree 4. Then $\tilde{X}$ is obtained by blowing-up five points $P_1,\ldots,P_5$ in general position in $\Pp_2$, and is embedded in $\Pp_4$ via the linear sistem of cubics passing through $P_1,\ldots,P_5$. The lines on $\tilde{X}$ correspond to the five blown-up points, to the lines on $\Pp_2$ connecting two such points, and to the unique conic containing $P_1,\ldots,P_5$; they all have self-intersection $-1$. Every hyperplane section of $\tilde{X}$ corresponds to a cubic curve in $\Pp_2$, passing through $P_1,\ldots,P_5$. If a hyperplane section is formed by two conics, the corresponding cubic is formed by a conic passing through four of the five points $P_1,\ldots,P_5$ and a line passing through the remaining one. Now, the given surface $\tilde{X}$ can be obtained in different ways as the blown-up of $\Pp_2$ over five points, namely one can choose in different ways the five lines on $\tilde{X}$ to  blow-down. We contend that after a suitable choice (of the five lines to blow-down) all the chosen conics $\cc_1,\ldots,\cc_5$ on $\tilde{X}$ correspond to conics in $\Pp_2$, each containing all but one point in the set $P_1,\ldots,P_5$; we omit this verification, which follows from the combinatorics of lines. Then let, for $i=1,\ldots,5$,  $F_i=0$ be an equation  for the conics corresponding to $\cc_i$ on the plane $\Pp_2$, so $F_i$ is a quadratic form in three variables. Let $G=0$  be an equation for the unique conic passing through $P_1,\ldots,P_5$, so $G$ too is a quadratic form. Again by Lemma \ref{riduzione}, 
the rational points on $\tilde{X}$, integral with respect to the strict transforms of $\cc_1,\ldots,\cc_5$ correspond to solutions $(x,y,z)\in\OS^3$ of the divisibility problem $F_i(x,y,z)|G(x,y,z)$. Then Theorem \ref{forme} applies and the Corollary is obtained.
\qed

\medskip

\noindent {\it Proof of Theorem \ref{DelPezzo2}}. Let $H_1,\ldots,H_4$ the hyperplane sections appearing in the statement, and suppose $H_4=\cc+\cc^\prime$ is reducible. Note that $H_i^2=H_i.H_j\, (=4)$ for $1\leq i,j\leq 3$ and $H_i.\cc\geq 1$  for all $i=1,2,3$. Then we can apply Theorem 1 of \cite{cz2}, concluding the proof.
\qed

%%%%%%%%%%%%%%%%%%%%
\section{Geometry  on certain open surfaces}\label{geometrica}

The aim of this section is two-fold: 
\smallskip

(1) We first want to calculate the generalized Albanese variety of the affine surfaces appearing in Theorem \ref{1} and in the Corollary to Theorem \ref{divisione} and prove the claim we made in the introduction: namely that such Albanese varieties are one dimensional for the surface appearing in Theorem \ref{1} and trivial for the one appearing in the Corollary; the latter is even simply connected (Theorem \ref{simplyconnected} below). 
\smallskip

(2) Second, we
want to classify   the possible infinite families of integral points on some of the open surfaces considered so far. 
In this respect, note that our results from section \ref{dimostrazioni} provide the degeneracy of the set of integral points on such surfaces $X$; this means that all but finitely many integral points lie on the union of finitely many curves on $X$. Now, Siegel's theorem on integral points on curves states that any such curve is rational, and moreover can be  parametrized either by $\A^1$ or by $\Gm$. 
 
We intend in this section to classify completely the non constant morphisms $\A^1\rightarrow X$ and $\Gm\rightarrow X$, for some of the surfaces   $X$ considered in the preceding section. 
Clearly, the most difficult case will concern the multiplicative group $\Gm$, i.e. the curves of vanishing Euler characteristic on $X$.

All the statements in this section are of geometric nature; in particular they hold over any field $k$ of characteristic zero, which one can suppose algebraically closed. 

\medskip
We begin by studing generalized Albanese varieties. Recall that given a smooth complete variety $\tilde{X}$ and a hypersurface $D\subset\tilde{X}$ with normal-crossing singularities (if any), putting $X=\tilde{X}\setminus D$, one can define the  generalised Albanese variety of $X$ (or quasi Albanese, for some authors) in the following way: it is a semi-abelian variety $G$ endowed with a morphism $\pi:X\to G$ with the following universal property: for every semi-abelian variety $G^\prime$ and every morphsim  $f: X\to G^\prime$, there exists a morphism $g: G \to G^\prime$  with $f=g\circ\pi$. The pair $(G,\pi)$ can also be defined analytically, by integration of the $1$-forms on $\tilde{X}$ with at most logarithmic poles along $D$. 

In this paper, all the varieties we consider are rational, so the Albanese varieties we are interested in will be linear tori. 

We start with the analysis of cubic surfaces; our reference is chapter IV from \cite{beau}. A smooth cubic surface $\tilde{X}$ contains twenty-seven lines, each having self-intersection $-1$; the maximal number of pairwise disjoint lines in $\tilde{X}$ is six. By blowing them down, one obtains the projective plane. Vice-versa, $\tilde{X}$ can be defined also as the blow-up of the projective plane over six points $P_1,\ldots,P_6\in\Pp_2$ in general position (i.e. no three of them on a line and not all of them on a conic). Let us denote by $\pi:\tilde{X}\rightarrow\Pp_2$ the corresponding blowing-up map. Then $\tilde{X}$ is embedded in $\Pp_3$ via the linear system of cubics passing through $P_1,\ldots,P_6$; the twenty-seven lines on $\tilde{X}$ correspond to the six exceptional divisors, the six conics containing five points among $P_1,\ldots,P_6$ and the fifteen lines joining two points $P_i,P_j$, for $1\leq i<j\leq 6$.
The six lines on the two hyperplane sections appearing in the statement can be obtained as follows: for each $i\in\Z/6\Z$, consider the line $L_i$ joining $P_i$ and $P_{i+1}$; denoting  by $\hat{L}_i$ its strict transform, the sum $\hat{L}_1+\hat{L}_3+\hat{L}_5$ will form a hyperplane section $H_1$ on $\tilde{X}$, while the other will be $H_2=\hat{L}_2+\hat{L}_4+\hat{L}_6$. Let us denote by $E_i$ the exceptional divisor above $P_i$; it is a line on $\tilde{X}$ intersecting both $L_i$ and $L_{i+1}$, but no other lines of the forms $L_j$. From this picture it immediately follows that   $L_i^2=-1$ for each $i\in\Z/6\Z$. Actually, each line on the cubic $\tilde{X}$ has self-intersection $-1$.
\smallskip

The above considerations will be crucial in the sequel; in particular in the proof  of the following:
\smallskip

\begin{theorem}\label{Albanesecubiche}
Let $\tilde{X}\subset\Pp_3$ be a smooth cubic hypersurface, $X\subset\tilde{X}$ be the open subset obtained by removing six lines on $\tilde{X}$ lying on two planes.
 The generalised Albanese variety of $X$ is a one-dimensional torus. 
\end{theorem}

\noindent{\it Proof of Theorem \ref{Albanesecubiche}}. Let $H_1,H_2$ be the hyperplane sections mentioned in the statement; each is defined by a linear equation $F_i(x_0,\ldots,x_3)=0$ in $\Pp_3$; the rational function $\pi:F_i/F_2$ induces a never vanishing regular function on $X$, still denoted by $\pi$, so a morphism $\pi: X\to \Gm$. 
We should prove that every other morphism $X\to\Gm$ factors through $\pi$. Letting $L_1,L_3,L_5$ be the three components of $H_1$ and $L_2,L_4,L_6$ the components of $H_2$, the divisor of $\pi$ (viewed as a rational fucntion on $\tilde{X}$) is 
\begin{equation*}
 (\pi)= L_1+L_3+L_5-L_2-L_4-L_6.
\end{equation*}
In particular, on the Picard group ${\rm Pic} (\tilde{X})$ we have the relation $L_1+L_3+L_5=L_2+L_4+L_6$. We shall prove that all the linear relation in the group generated by $L_1,\ldots,L_6$ derive from the above one, so such group is free-abelian of rank $5$. For this task, consider any non-trivial linear combination of  $L_1,\ldots,L_6$ omitting  one term, for instance $L_6$. So take an index $1\leq i\leq 5$, integers $a_1,\ldots,a_i$, with $a_i\neq 0$, such that $a_1L_1+\ldots+a_iL_i=0$ in ${\rm Pic}(\tilde{X})$. Taking the intersection product with $E_i$, we have $(a_1L_1+\ldots+a_iL_i).E_i=a_i\neq 0$, which shows that the linear combination under consideration deos not vanish in the Picard group.
 
Let now $f:X\to\Gm$ be another morphism. The divisor of $f$ must also have its support in $L_1\cup\ldots\cup L_6$. Let us write
\begin{equation*}
 (f)=\sum_{i\in\Z/6\Z}m_i L_i-\sum_{j\in \Z/6\Z}n_jL_j
\end{equation*}
 where the two functions $\Z/6\Z\ni i\mapsto m_i\in\{0,1,\ldots\}$ and $\Z/6\Z\ni i\mapsto n_i\in\{0,1,\ldots\}$ have disjoint supports. Since the divisor class of $\sum_{i\in\Z/6\Z}m_iL_i-\sum_{j\in \Z/6\Z}n_jL_j $ in the Picard group vanishes, by the above consideration we must have  $\sum_{i\in\Z/6\Z}m_iL_i-\sum_{j\in \Z/6\Z}n_jL_j=N(L_1+L_3+L_5-L_2-L_4-L_6) $ for some integer $N\in\Z$. Then the function $f/\pi^N$ has a trivial divisor, so is a constant $\lambda\neq 0$. In other words, $f=\lambda\cdot \pi^N$ is obtained from $\pi$ by composition with an endomorphism of $\Gm$, which proves that the Albanese variety of $X$ is isomorphic to $\Gm$.
\qed

\medskip

Let us now consider the surface $X$ appearing in Corollary \ref{rettescoppiate}; our goal is to prove Theorem \ref{simplyconnected}, stating that $X$ is simply connected. Recall that $X=\tilde{X}\setminus D$, where $\tilde{X}$ is the plane blown-up on three points $P_1,P_2,P_3$, and $D$ is the strict transform of four lines $L_1,\ldots,L_4$, with $P_i\in L_i$ for $i=1,2,3$. 
Let $D_i$ be the strict transform of $L_i$ (for $i=1,2,3$), and let $E_i$  (for $i=1,2,3$) be the exceptional divisor crossing $D_i$. Finally, let $Y\subset X$ be the surface 
\begin{equation*}
 Y:=\tilde{X}\setminus (D_1\cup\ldots\cup D_4\cup E_1\cup\ldots\cup E_3)=X\setminus(E_1\cup E_2\cup E_3).
\end{equation*}
It is clear that $Y\simeq \Pp_2\setminus(L_1\cup\ldots\cup L_4)$. Now, to prove Theorem \ref{simplyconnected} we shall use the following facts, stated as lemmas \ref{aperto}, \ref{Z3}, \ref{nociclici}:

\begin{lemma}\label{aperto}
 Let $X$ be a (connected) complex manifold, $Z\subset X$  a proper closed complex submanifold, $Y=X\setminus Z$ its complement, $p\in Y$ a point. Then the inclusion $i:Y\hookrightarrow X$ induces a surjective homomorphism $i_*:\pi_1(Y,p)\twoheadrightarrow\pi_1(X,p)$ between the corresponding fundamental groups.
\end{lemma}

\begin{proof} 
 We shall use the functorial property of the fundamental group. Let $\pi:X^\prime\to X$ be the universal cover of $X$, $\Gamma\simeq \pi_1(X,p)$ be the deck-transformation group, and let $Y^\prime:=\pi^{-1}(Y)$ be the pre-image of $Y$ in $X^\prime$. Then $Y^\prime=X^\prime\setminus Z^\prime$, where $Z^\prime:=\pi^{-1}(Z)$ is a complex analytic sub-variety. In particular $Z^\prime$ has real codimension at least two, so $Y^\prime$ is connected and $\Gamma$ is the automorphism group of the cover $Y^\prime\to Y$. Letting $\pi^{\prime}:Y^{\prime\prime}\to Y^\prime$ be the universal cover of $Y^\prime$, we obtain by composition the universal cover $\pi\circ\pi^\prime:Y^{\prime\prime}\to Y$ of $Y$. Then the group $\Gamma$ is realised as a quotient of the deck-transformations group of $Y^{\prime\prime}\to Y$, i.e. of $\pi_1(Y,p)$, as wanted. It is easy to see that the surjective homomorphism $\pi_1(Y,p)\to\Gamma=\pi_1(X,p)$ so constructed coincides with the homomorphism $i_*$ (but actually we do not even need this fact in the sequel).
\end{proof}

An alternative, more concrete proof, is based on the following fact: {\it Let  $X$ be a real manifold  of dimension $\geq 2$, $Z\subset X$ a codimension two subvariety. Let $p\in X\setminus Z$ be a point, $\gamma: [0,1]\to X$ a loop with base point $p$. Then $\gamma$ is homotopic to a loop in $X\setminus Z$ (with same base point).} This fact is well-known, and appears for instance a Lemma 3.12 in \cite{bl}. 

\begin{lemma}\label{Z3}
 The fundamental group of the complement of four lines in general position in $\Pp_2$ is isomorphic to $\Z^3$. 
\end{lemma}

\begin{proof}
 This is a well-known theorem of Zariski,  proved \footnote{Zariski stated a much more general result in \cite{Zariski},  on the fundamental group of the complement of any (reducible) curve with normal crossing singularities; as remarked for instance by Mumford in the notes to the latest edition of Zariski's book  \cite{Zariski}, the proof of this general statement is incorrect, and was settled only later by Abhyankar. Neverthless, in the case of an arrangement of lines, considered in Corollary \ref{Z3}, Zariski's proof is sound.} in \cite{Zariski}, Chap. VIII.
\end{proof}

\begin{lemma}\label{nociclici}
 The surface $X$ admits no finite cyclic unramified connected cover of degree $>1$. 
\end{lemma}

\begin{proof}
Let $X^\prime\to X$ be a cyclic unramified cover of degree $d>1$. By ``Riemann's Existence Theorem \footnote{proved in its general form by Grauert and Remmert}, see Th\'eor\`eme 5.1 in \cite{sga1}, $X^\prime$ has a structure of algebraic variety  such that the covering map is algebraic. Then the function field of $X^\prime$  is obtained by taking the $d$-th root of a rational function $f\in\C (\tilde{X})^*$ on $\tilde{X}$. The property that $X^\prime\to X$ is unramified means that all zeros and poles of $f$ in $X$ have multiplicity divisible by $d$. This means precisely that: for each irreducible curve $\cc\subset\tilde{X}$  {\it outside} $D_1\cup\ldots\cup D_4$, the multiplicity of $\cc$ in the divisor $(f)$  is divisible by $d$. Considering the blow-up map $\pi:\tilde{X}\to\Pp_2$, which is birational, the function $f$ defines a rational function $(\pi^{-1})^*(f)=:g$ on $\Pp_2$. Hence the function field of $X^\prime$ is obtained from the function field $\C(\Pp_2)=\C(x,y)$ by adding the $d$-th root of $g$. Using the fact that the Picard group of $\Pp_2$ is cyclic we can always multiply $g$ by a suitable $d$-th power having all its poles in $L_4$ obtaining a new function, still denoted by $g$, which is regular outside $L_4$. Hence $g$ can be viewed as a polynomial in $x,y$, after identifying $\Pp_2\setminus L_4\simeq \A^2$. Let $g_i=0$ be a (linear) equation for $L_i\cap\A^2$ ($i=1,2,3$). The condition on the divisor of $f$ (in $\tilde{X}$) can be stated in terms of $g$ by saying that each irreducible factor of $g$ distinct from $g_i$ occurs with multiplicity divisible by  $d$. This will not be enough to conclude, but now observe that if some $g_i$ occured in the factorisation of $g$ with multiplicity not divisible by $d$, then the entire pull-back of $L_i$ would lie in the ramification locus of the cover $\tilde{X^\prime}\to\tilde{X}$, contrary to our hypothesis (recall that we are assuming that $E_i$ does not lie on the ramification locus). Then we have that each irreducible factor of $g$ occurs with multiplicity divisible by $d$, so $g$ is a perfect $d$-th power, and the same is true of $f$. But this means that the cover $X^\prime\to X$ is disconnected, concluding the proof. 
\end{proof}

{\it Proof of Theorem \ref{simplyconnected}}. We shall use the notation of Corollary \ref{rettescoppiate}, already used in the proofs above. Let $Y\subset X$ be the 
open set $X\setminus(E_1\cup E_2\cup E_3)\simeq \Pp_2\setminus(L_1\cup\ldots\cup L_4)$.
Then by Lemma \ref{aperto}, the fundamental group of $X$ is a quotient of the fundamental group of $Y$. Now, Lemma \ref{Z3} assures it is a quotient of the group $\Z^3$; to prove that such a quotient is trivial, it is sufficient to prove that it admits no non-trivial finite cyclic quotients. Geometrically, this means that $X$ admits no non-trivial (finite) cyclic  cover, which is the content of Lemma \ref{nociclici}, finishing the proof.
$\square$
\smallskip

An alternative proof of Theorem \ref{simplyconnected},   avoiding the use Lemmas \ref{Z3} and \ref{nociclici}, consists in noticing that the topology of such surfaces is independent of the particular choice of the lines $L_1,\ldots,L_4$ and the points $P_1,\ldots,P_3$. Then, it suffices to prove the Theorem for one single such surface; this can be done for instance via Lemma \ref{aperto}, by proving that each generator of the fundamental group of $Y=X\setminus (E_1\cup E_2\cup E_3)$ becomes homotopically trivial in $X$.

\medskip

Our last goal in this section is to conclude the proof of Theorem \ref{teounits}, classifying the possible infinite families of solutions to equation (\ref{sunits}).
This amounts to finding  the non-constant morphisms from $\Gm$ to open sets in the Hirzebruch surface considered in the context of parametric $S$-unit equations:
\medskip

\begin{theorem}\label{linesHirzebruch}
Let $f(T),g(T),h(T)\in k[T]$ be three polynomials of the same degree, pairwise coprime. Let $X\subset\Gm^2\times\A^1$ be the surface defined by equation (\ref{sunits}), where $(u,v)$ are coordinates in $\Gm^2$, and $t$ is the coordinate  in $\A^1$. Then every morphism $\A^1\to X$ is constant. Let $\Gm\ni x\mapsto (u(x),v(x),t(x))\in X$ be a morphism; then at least one among the three functions $u,v,u/v$ is constant. For every zero $t_0$ of $f(t)g(t)h(t)$, the fiber of $t_0$ in $X$ is a curve isomorphic to $\Gm$, so in particular there exists a non constant morphism $(u,v,t):\, \Gm\to X$ with $t(x)\equiv t_0$. 
\end{theorem}

\noindent{\it Proof}. Let $\Gm\ni x\mapsto (u(x),v(x),t(x))\in X$ be a non constant morphism. Suppose  first that the regular function $t:\Gm\rightarrow\A^1$ is constant, equal to $t_0$. Then the curve in $\Gm^2$ defined by the equation $f(t_0)u+g(t_0)v=h(t_0)$ has Euler characteristic $1$ (it is isomorphic to $\Pp_1\setminus\{0,1,\infty\}$) unless $f(t_0)g(t_0)h(t_0)=0$, in which case it is isomorphic to $\Gm$ (so its Euler characteristic vanishes). In the first case, the morphism would be constant; in the second case, if  $f(t_0)=0$, then by hypothesis $g(t_0)\neq 0\neq h(t_0)$, so necessarily $v(x)\equiv -h(t_0)/g(t_0)$ would be constant; if $g(t_0)=0$, then $u$ would be constant, for the same reason; finally, if $h(t_0)=0$, then $u(x)/v(x)$ would constant, be equal to $-g(t_0)/f(t_0)$. Consider now the case that the function $t$ is non constant. It might have one or two poles. Let $p \in\{0,\infty\}$ be a pole of $t$; if $p$ is a pole of $u$ or $v$, after dividing both sides in (\ref{sunits}) by $u$ (resp. $v$), we obtain another equation of the same type, namely $f(t)+g(t)(v/u)=h(t)(1/v)$ (or $f(t)(u/v)+g(t)=h(t)(1/v)$), where the new functions $(v/u), (1/u)$ (or $(u/v),(1/v)$) would have no pole in $p$. Hence we are reduced to the case where one pole of $t$, say $\infty$, is not a pole of $u,v$; if we prove that in this particular case $u$ or $v$ is constant, we would have proved in general that one among $u,v,u/v$ is so. Then, suppose that $u(x)=\lambda x^{-a}, v(x)=\mu x^{-b}$, for $\lambda,\mu\in k^*$, $a\geq 0$, $b\geq 0$.  Then the equation becomes
\begin{equation*}
 \frac{\lambda f(t(x))}{ x^a} + \frac{\mu g(t(x))}{ x^b} = h(t(x)).
\end{equation*}
Comparing the order of pole at infinity, we deduce that $\min\{a,b\}=0$, as wanted.
\qed

\smallskip

We note that, for $d=\deg(f)=\deg(g)=\deg(h)=1$, there always exists a non constant morphism $\Gm\to X$, with a non constant $t$, i.e. which remains non constant under composition with the canonical projection $X\to\A^1$. In fact, after change of variable in $t$ (i.e. composition with an automorphism of $\A^1$) equation (\ref{sunits}) with $d=1$ takes the form 
\begin{equation*}
tu+(1-t)v=t+a
\end{equation*}
for a suitable $a\in k^*$. Then, the map $\Gm\ni x\mapsto (u,v,t)=(x,x,x-a)$ is a non constant morphism. For $d\geq 2$ and generic choice of the polynomials $f(t),g(t),h(t)$, the image of every morphism $\Gm\to X$ will be contained in a fibre $t=t_0$ for the projection $X\to\A^1$; for special choices of $f(t),g(t),h(t)$ there will be exactly one exception. 

\section{Density of integral points on certain surfaces}\label{controesempi}

In this paragraph, $k$ will again denote a number field. Our aim is to prove that some of the main results are in a sense best possible, by proving that certain surfaces posses Zariski-dense sets of integral points.  

The main ideas and tools in this section originated in the paper \cite{beukers}, and were vastly generalized by Hassett and Tschinkel in \cite{ht}. In particular, we shall repeatedly use the following lemma, which is a particular case of   Theorem 2.3 of \cite{beukers}:

\begin{lemma}\label{friz}
 Let $k$ be a number field, $\OS\subset k$ be a ring of $S$-integers such that the group $\OS^*$ is infinite. Let $\cc\subset\Pp_2$ a smooth conic defined over $k$, $A,B\in\cc(\bar{k})$ be rational points on $\cc$ (possibly $A=B$), such that  the divisor $A+B$ on $\cc$ is defined over $k$. Then the set of $S$-integral points on the open curve $\cc\setminus\{A,B\}$ is either empty or infinite.
\end{lemma}

\hfill $\square$
\medskip

The idea of the proofs of all results in this section is to produce a conic fibration on the relevant surfaces, and apply Lemma \ref{friz} to prove that infinitely many conics in such fibrations possess infinitely many integral points. This is the same idea of Hassett and Tschinkel.

We begin by proving that the condition $d\geq 6$ in Theorem \ref{ennagono} is really ne\-ces\-sary. In fact, the next result proves that the conclusion of Theorem \ref{ennagono} fails if $d=5$; the case $d\leq 4$ is easier and its proof will be left to the reader.

\smallskip
 
\begin{theorem}\label{5}
 Let $L_1,\ldots,L_5$ be lines on $\Pp_2$, defined over $k$, in general position. Let, for $i\in\Z/5\Z$, $P_i$ be the point of intersection $L_i\cap L_{i+1}$. Let $\tilde{X}\rightarrow\Pp_2$ be the blow-up of the plane over $P_1,\ldots,P_5$ and denote by $\hat{L}_i$ the strict transforms of the lines $L_i$. Then for a suitable finite set $S$ of places of $k$, the integral points on the surface $\tilde{X}\setminus(\hat{L_1}\cup\ldots\cup\hat{L}_5)$ are Zariski-dense.
\end{theorem}

\begin{proof}
We shall actually prove a little more, namely the density of integral points on $\tilde{X}$ with respect to the divisor $\hat{L}_1+\ldots+\hat{L}_5+E_5$, where $E_5$ is the exceptional divisor over $P_5$ (in other words: we do not need to blow up the points $P_5$). First of all, let us enlarge the set $S$ (if necessary) so that the group of units $\OS^*$  is infinite and the configuration of lines $L_1+\ldots+L_5$ has good reduction outside $S$. By this, we mean that their reductions modulo every palce outside $S$ are still in general position. Denote by $M$ the line joining $P_1$ to $P_4$.

\noindent Consider the pencil $\Pp$ of conics passing through $P_1,\ldots,P_4$. It is parametrised by the projective line $\Pp_1$; those conics which do not reduce (modulo any place outside $S$) to the reducible conic  $\cc_1:=L_1+L_4$ nor to the conic $\cc_2:=L_2+M$ are parametrised by $\Pp\setminus\{\cc_1,\cc_2\}\simeq\Gm$; in particular, their exist infinitely many such conics. Let $\Pp^\prime\subset\Pp(k)$ be this  set of conics.  For every such conic $\cc\in\Pp^\prime$, let $A_\cc, B_\cc$ be the second point of intersection with $L_5$ and $L_4$ respectively (they might coincide, in which case $A_\cc=B_\cc=P_4$; also, we might have $A_\cc=P_5$ or $B_\cc=P_3$, in case $\cc$ is tangent to $L_5$ at $P_5$ or to $L_4$ at $P_4$).  We pause to prove that all the   integral points on $\cc$ with respect to $A_\cc,B_\cc$ give rise to integral points on $\tilde{X}$ with respect to $ $.  Hence, for every such conic $\cc$ the point $P_2$ is integral with respect to $\hat{L}_1+\ldots+\hat{L}_5+E_5$. In fact, let $Q$ be a point on $\cc$ which is integral with respect to $A_\cc, B_\cc$, and let $\hat{Q}$ be its corresponding point on $\tilde{X}$. We have to prove that $\hat{Q}$  does not reduce to  $\hat{L}_1+\ldots+\hat{L}_5+E_5$. 
For this purpose, we  note that no element $\cc\in\Pp^\prime$  is tangent to $L_2$ or $L_3$ at the point $P_2$, nor becomes tangent after reducing modulo any place outside $S$;  it is not tangent (neither reduces to a tangent one) to $L_1$ (at $P_1$ or $P_5$) nor to $L_3$ (at $P_3$ or $P_4$). Then $\hat{Q}$ cannot reduce to $\hat{L}_1+\ldots\hat{L}_4$; the integrality of the point $Q$ with respect to $A_\cc,B_\cc$ just means that it $\hat{Q}$ does not reduce to $\hat{L}_4+\hat{L}_5+E_5$. Hence, to prove Theorem \ref{5} it is sufficent to prove that such conics have infinitely many integral points. By Lemma \ref{friz}, it is sufficient to find one integral points. Now, by our assumptions on $S$, the point $P_2$ is  integral with respect to $L_4+L_5$, so in particular with respect to $A_\cc+B_\cc$, concluding the proof.
\end{proof}

\begin{corollary}\label{ctforme}
 Let $F_1,\ldots,F_5$ be linear forms in three variables, in general position, defined over $k$. There exists a non degenerate quadratic form $G$, defined over $k$, and a finite set $S$ of places of $k$ such that the points $(x,y,z)\in\OS^3$ with $F_i(x,y,z)|G(x,y,z)$ are Zariski-dense in $\Pp_2$.  
\end{corollary}

 \begin{proof}
  Putting, as above $P_i=F_i^{-1}(0)\cap F_{i+1}^{-1}(0)$, for $i\in\Z/5\Z$, let $\cc$ be the (unique smooth) conic passing through $P_1,\ldots,P_5$. Let $G=0$ be a homogeneous equation for $\cc$. Define again $X$ to be the complement of the strict transform of the lines $F_i=0$ in the blow-up of the plane over $P_1,\ldots,P_5$. The integral points on $X$ correspond to the solution of the above divisibility problem.
Now, by Theorem \ref{5}, such integral points are Zariski-dense, under a suitable ring extension of $\OS$.
 \end{proof}

From the general position assumption, the points $(x,y,z)\in\OS^3$ with coprime coordinates and satisfying  $F_i(x,y,z)|G(x,y,z)$, will also satisfy
\begin{equation*}
 F_1(x,y,z)\cdots F_5(x,y,z)|G^2(x,y,z). 
\end{equation*}

So, after the above Corollary, we can obtain a dense set of integral points where a degree-five homogeneous form divides a form of degree four.
\medskip

We now turn our attention to cubic surfaces again; we want to prove that Theorem 1 is best possible; more precisely, we prove Theorem \ref{ctex1}, asserting that: (1) if we remove just five lines on two hyperplane sections from a cubic surfaces, the integral points are Zariski-dense (in a suitable ring of $S$-integers); (2) one can always remove nine lines in such a way that the integral points on the complement are still dense.

\noindent{\it Proof of Theorem \ref{ctex1}}. Let us begin by the first assertion. We recall from the preceeding section that a cubic surface can always be realized as the blown-up of six points $P_1,\ldots,P_6$ in general position on the plane. Up to enlarging $k$ and $S$ we can suppose they are all defined over $k$ and they remain in general position when reduced modulo every place outside $S$. The twenty-seven lines correspond to the six blown-up points, to the fifteen lines joining two of them, to the six conics passing through five of them. Suppose we are given five lines on two hyperplane sections: then either the two hyperplane sections share a common line, or the union of the two hyperplane sections consists of six lines (and we are considering just five of them). 

Let us consider te first case, where a line is common to the two hyperplane sections. We can realize the first hyperplane section as the image of the lines $L(P_1,P_2), L(P_3,P_4), L(P_5,P_6)$; the second one as the image of the lines $L(P_3,P_4)$ (in common with the preceding hyperplane section), $L(P_2,P_6)$, $L(P_1,P_5)$. Then consider the pencil $\Lambda$ of conics through $P_1,P_2,P_5,P_6$; those which do not reduce, modulo any place outside $S$, to the singular conic $\cc_1:=L(P_1,P_2)+L(P_5,P_6)$ nor to $\cc_2:=L(P_2,P_6)+L(P_1,P_5)$ correspond to the integral points on $\Lambda\setminus\{\cc_1,\cc_2\}\simeq\Gm$, so there are infinitely many of them. Fix one such  conic $\cc$; its image on the cubic surface intersects the union of the two hyperplane sections on the (image of the) line $L(P_3,P_4)$, so on a divisor of the form $A+B$, consisting on one or two points ($A$ might coincide with $B$). Each such conic posses at least one integral point with respect to this divisor, namely $P_1$\footnote{note, however, that the point $P_1$ on $\Pp_2$ corresponds to a line on the cubic surface; this line intersects the image of $\cc$ on one point, corresponding on the tangent direction of $\cc$ at $P_1$; this direction cannot be the direction of the line $L(P_1,P_2)$, nor can reduce to this modulo any place outside $S$}. Then, by Lemma \ref{friz}, $\cc$ contains infinitely many points which are integral with respect to $A+B$. We then obtain infinitely many curves on the surface, each admitting infinitely many integral points, thus providing a Zariski-dense infinite set as wanted.

In the second case, we argue in a very similar way. we can suppose, as before, that the first hyperplane section is given by the lines $L(P_1,P_2), L(P_3,P_4), L(P_5,P_6)$;  let the  second one be defined by $L(P_2,P_3), L(P_4,P_5)$, $L(P_6,P_7)$. Consider the complement of the first five lines, i.e. all the above but $L(P_5,P_6)$; we claim that, even after removing the two lines corresponding to the blown-up points $P_1,P_6$, the integral points are Zariski-dense. To prove this, consider the pencil $\Lambda$ of conics passing through $P_2,\ldots,P_4$; for every such conic $\cc\in\Lambda$, let $A+P_2$ be its intersection with   $L(P_1,P_2)$, $B+P_5$ its intersection with $L(P_5,P_6)$; we do not exclude that $A=B$, nor $A=P_2$, nor $B=P_5$. Integral points on $\cc\setminus\{A,B\}$ give rise to integral points on the surface we are interested in. Note that $\Lambda\simeq\Pp_1$, and the conics in $\Lambda(k)$ which do not reduce to $L(P_2,P_3)+L(P_4,P_5)$ nor to $L(P_3,P_4)+L(P_2,P_5)$ modulo any place outside $S$ correspond to rational points on the line $\Lambda$ which are integral with respect to two points, i.e. to integral points on $\Gm$, hence they are Zariski-dense, over a suitable ring of $S$-integers. By Lemma \ref{friz}, each such conic   contains infinitely many integral points, with respect to the divisor $A+B$,  since it contains the point   the integral $P_3$ (or $P_4$). Then the integral points on the surface are dense, so this case is settled too.
 
Let us now prove the second sentence, namely that one can sometimes remove up to nine lines, and still obtain a Zariski dense set of integral points. Again, consider the plane blown up at six points $P_1,\ldots,P_6$ in general position. The three lines $L(P_1,P_2), L(P_3,P_4), L(P_5,P_6)$, as mentioned, give rise to corresponding lines on the cubic surface; adding the six lines which correspond to the blown-up points, we obtain a configuration of nine lines whose complement is isomorphic to the complement of the mentioned three lines on the plane. This complement is then isomorphic to $\Gm^2$, whose instegral points are indeed Zariski-dense, over a suitable ring of $S$-integers, concluding the proof.
\qed
\medskip

{\bf Concluding remarks}. As mentioned, in  \S \ref{dimostrazioni} we omitted the proofs in the Nevanlinna theoretic setting, since they are essentially the same as in the arithmetic case, up to the replacement of Schmidt's Subspace Theorem by Cartan's Second Main Theorem. 
In the present section, however, as far as density is concerned, we can go further   then the exact analogue of the arithmetic case:  one can prove that the open surfaces considered in Theorems \ref{ctex1}, \ref{ctforme}, not only contain Zariski-dense entire curves, but are even holomorphically dominated by $\C^2$. This last fact could be proved by combining the above technique of producing  conic fibrations with recent work of Buzzard and Lu \cite{bl}.
%%%%%%%%%%%%%%%%%%%%%%%%%%%%%%

\bigskip
{\bf Acknowledgments}. Part of this paper was written during the thematic program ``Arithmetic Geometry, Hyperbolic Geometry and related topics'' held at Fields Institute in Toronto, in the fall 2008. We thank the organizers  and the Fields Institute for the invitation and for financial support.

\vfill

Pietro Corvaja\hfill Umberto Zannier

Dip. di Matematica e Informatica\hfill  Scuola Normale Superiore

Via delle Scienze\hfill Piazza dei Cavalieri, 7

33100 - Udine (ITALY)\hfill  56100 Pisa (ITALY)

pietro.corvaja@dimi.uniud.it\hfill u.zannier@sns.it

\end{document}